\documentclass[reqno,tbtags,intlimits,a4paper,oneside,12pt]{amsart}
\usepackage[cp1251]{inputenc}%
\usepackage[english]{babel}
\usepackage{amssymb,upref,calc,url}
\usepackage[mathscr]{eucal}
\usepackage{graphicx}
\usepackage{amsmath,amscd,amsthm,verbatim}
\usepackage[all]{xy}
\usepackage[in]{fullpage}
\theoremstyle{plain}
\newtheorem{thm}{Theorem}[section]
\newtheorem{lem}[thm]{Lemma}
\newtheorem{cor}[thm]{Corollary}
\newtheorem{utv}[thm]{Proposition}

\theoremstyle{definition}
\newtheorem{dfn}{Definition}[section]
\newtheorem{ex}[thm]{Example}
\begin{document}
\title{Elliptic function of level 4}
\author{Elena~Yu.~Bunkova}
\address{Steklov Mathematical Institute, Russian Academy of Sciences}
\email{bunkova@mi.ras.ru}
\thanks{This work is supported by the Russian Science Foundation under grant 14-50-00005.}

\maketitle

\begin{abstract}
The work is dedicated to the theory of elliptic functions of level $n$.
An elliptic function of level $n$ determines a Hirzebruch genus that is called elliptic genus of level $n$. 
Elliptic functions of level $n$ are also interesting as solutions of Hirzebruch functional equations.

The elliptic function of level $2$ is the Jacobi elliptic sine.
It determines the famous Ochanine--Witten genus.
It is the exponential of the universal formal group of the form
\[
 F(u,v)=\frac{u^2 -v^2}{u B(v) - v B(u)}, \quad B(0) = 1.
\]
The elliptic function of level $3$ is the exponential of the universal formal group of the form
\[
 F(u,v)=\frac{u^2 A(v) -v^2 A(u)}{u A(v)^2 - v A(u)^2}, \qquad A(0) = 1, \quad A''(0) = 0.
\]
In this work we have obtained that the elliptic function of level $4$ 
is the exponential of the universal formal group of the form
\[
 F(u,v)=\frac{u^2 A(v) -v^2 A(u)}{u B(v)-v B(u)}, \text{ where } A(0) = B(0) = 1, 
\]
and for $B'(0) = A''(0) = 0, A'(0) = A_1,  B''(0) = 2 B_2$ the relation holds
\[
(2 B(u) + 3 A_1 u)^2 = 4 A(u)^3 - (3 A_1^2 - 8 B_2) u^2 A(u)^2. 
\]
To prove this result we have expressed the elliptic function of level $4$
in terms of Weierstrass elliptic functions.
\end{abstract}

\section{Introduction}

\begin{dfn}
We call a \emph{non-degenerate elliptic function of level $n$}
a meromorphic function $f_n(x)$ in $\mathbb{C}$ with quasiperiodic properties
\[
 f_n(x + 2 \omega_1) = f_n(x), \qquad f_n(x + 2 \omega_2) = \epsilon_n f_n(x), \qquad \omega_2 / \omega_1 \notin \mathbb{R},
\]
where $\epsilon_n$ is a primitive $n$-th root of unity,
that has a single simple pole on in the periods parallelogram with generators $(2 \omega_1, 2 \omega_2)$. 
\end{dfn}

From elliptic functions theory (see \cite{WW, Iso}) it follows that such a function has a single simple zero on the periods parallelogram.
It is convenient to assume that this zero is at the origin, and we will use the normalization
\[
 f_n(0) = 0, \quad f_n'(0) = 1.
\]

The interest in elliptic functions of level $n$ is related to the fact that such functions determine elliptic Hirzebruch genera of level $N$,
where $n \mid N$ (see \cite{Hirtz, Hirzebruch-88}).
In particular, the function $f_2(x)$ is the elliptic sine and determines the famous Oshanine--Witten elliptic genus \cite{Osh}.  
In this work we propose an approach to the study of elliptic functions of level $n$,
which uses that they are elliptic functions with periods $2 \omega_1$, $2 n \omega_2$. We demonstrate this approach for $n = 2, 3, 4$.

Non-degenerate elliptic functions of level $n$ exist for any $n \geqslant 2$.
They are expressed in terms of the Baker--Akhiezer function $\Phi(x;z)$: we have
\begin{equation} \label{fKr}
f_{Kr}(x) = {\exp(\alpha x) \over \Phi(x;z)} =  { \sigma(x) \sigma(z) \over \sigma(z - x)} \exp(\alpha x - \zeta(z) x),
\end{equation}
and $f_n(x) = f_{Kr}(x)$
for
$z = 2 {k \over n} \omega_1$, $\alpha = - 2 {k \over n} \zeta(\omega_1) + \zeta(z)$, $0 < k < n$, $\text{GCD}(k,n) = 1$.
Here $\sigma(x)$ and $\zeta(x)$ are Weierstrass functions with quasiperiods $2 \omega_1$, $2 \omega_2$.
From this expression it follows that the simple pole of $f_n(x)$ is at $x = z = 2 {k \over n} \omega_1$.  

An expression of the form \eqref{fKr} for functions determining elliptic denera of level $N$ was obtained in \cite{Krichever-90}.
Herewith $z = 2 {k_1 \over N} \omega_1 + 2 {k_2 \over N} \omega_2$ is a point of order $N$ with respect to the lattice with generators $(2 \omega_1, 2 \omega_2)$,
and $\alpha = - 2 {k_1 \over N} \zeta(\omega_1) - 2 {k_2 \over N} \zeta(\omega_2) + \zeta(z)$.
By replacing the generators of the lattice $(2 \omega_1, 2 \omega_2)$ we come to an equivalent expression of the form \eqref{fKr} where $k_2 = 0$,
that is
$z = 2 {k_1 \over N} \omega_1 = 2 {k \over n} \omega_1$, $\alpha = - 2 {k \over n} \zeta(\omega_1) + \zeta(z)$, $0 < k < n$, $\text{GCD}(k,n) = 1$.

The function $f_{Kr}(x)$ can be considered as a four-parametric function with parameters $\alpha, \beta, \gamma, \lambda$ (see Corollary \ref{Kr}).
The non-degeneracy condition is $\Delta = \lambda^3 - 27 (4 \beta^3 - \lambda \beta - \gamma^2)^2 \ne 0$.
We show that each of the functions $f_2(x)$, $f_3(x)$, $f_4(x)$ can be presented as $f_{Kr}(x)$,
where in each case the parameters belong to a two-parametric subset in the space of parameters $(\alpha, \beta, \gamma, \lambda) \in \mathbb{C}^4$.
By removing the restriction $\Delta \ne 0$ we obtain \emph{elliptic functions of level $n$}.

The elliptic function of level $n$ for $n \mid N$ is a solution of the \emph{$N$-th special Hirzebruch functional equation}
(see \cite{Hirtz},  4.5 and Appendix III)
\begin{equation} \label{fe1}
	\sum_{i = 1}^{N} \prod_{j \ne i} { 1 \over f(x_j - x_i)} = 0.
\end{equation}
(``Special'' refers to a Hirzebruch equation with zero on the right hand side. 
In the general case, the right hand side of a Hirzebruch equation is a constant.)
In this work, we consider solutions of this equation with initial conditions $f(0) = 0, f'(0) = 1$.
Such solutions determine special $\mathbb{C}P^{N-1}$-multiplicative Hirzebruch genera.
They are rigid on $\mathbb{C}P^{N-1}$ (see \cite{BBZam, BBHi2, BN}).

For topological applications it is important to know the coefficient rings of universal formal groups
whose exponentials are elliptic functions of level $n$.
For $n = 2,3$ such formal groups are described in~\cite{BBU}. In the general case it is known (see \cite{Buc90}) that
these formal groups are specializations of 
\emph{Buchstaber formal group}
\begin{equation} \label{Buch}
F(u,v) =\frac{u^2A(v)-v^2A(u)}{uB(v)-vB(u)}, \text{ where } A(0) = B(0) = 1.
\end{equation}

In theorem \ref{Tmain} we have found the necessary and sufficient conditions for the formal series $A(u)$ and $B(u)$
so that the exponential of the formal group~\eqref{Buch} is an elliptic function of level $4$.

The author is grateful to V.\,M.\,Buchstaber for raising the problem and
to S.\,O.\,Gorchinskiy for helpful remarks during the preparation of this work for publication.

\section{Weierstrass and Baker--Akhiezer functions}

An \emph{elliptic function} is a meromorphic function $f(x)$ on $\mathbb{C}$ with periodic properties
\begin{equation} \label{per}
 f(x + 2 \omega_1) = f(x), \qquad f(x + 2 \omega_2) = f(x), \qquad \omega_2 / \omega_1 \notin \mathbb{R}.
\end{equation}

The points $2 n \omega_1 + 2 m \omega_2$, where $n,m \in \mathbb{Z}$,
form the periods lattice $\langle 2 \omega_1, 2 \omega_2 \rangle$.
The properties \eqref{per} are equivalent to $f(x + \omega) = f(x)$ for any
$\omega \in \langle 2 \omega_1, 2 \omega_2 \rangle$.
The latter property does not depend on the choice of generators $(2 \omega_1, 2 \omega_2)$ in the periods lattice.
The function $f(x)$ with properties \eqref{per}
can be regarded as a function on the Jacobian of the elliptic curve $\mathbb{C}/\langle 2 \omega_1, 2 \omega_2 \rangle$.

\emph{The Weierstrass function $\wp(x)$} (see \cite{WW, Iso}) is defined as an elliptic function 
with poles only in lattice $\langle 2 \omega_1, 2 \omega_2 \rangle$  points
and a series expansion at the origin of the form
\begin{equation} \label{nu}
 \wp(x) = {1 \over x^2} + O(x).
\end{equation}

The functions $\wp(x)$ and $\wp'(x)$ are connected by \emph{Weierstrass equation}
\begin{equation} \label{wp}
\wp'(x)^2 = 4 \wp(x)^3 - g_2 \wp(x) - g_3.
\end{equation}
The discriminant of the corresponding elliptic curve is $\Delta = g_2^3 - 27 g_3^2$.
In the conditions of \eqref{per} the curve is non-singular, i.e. $\Delta \ne 0$. 

It is convenient to take the invariants $g_2, g_3$ as parameters of the $\wp$-function.
In this case we use the notation $\wp(x; g_2, g_3)$. The equation \eqref{wp} with initial conditions \eqref{nu}
determines the coefficients of the series expansion of
$\wp(x; g_2, g_3)$ as polynomials in $(g_2, g_3) \in \mathbb{C}^2$.
This allows to define Weierstrass functions corresponding to the degenerate case $\Delta = 0$.

The function $\wp(x)$ is even. 

\emph{Weierstrass function $\zeta(x)$} is a meromorphic function on $\mathbb{C}$ determined by the conditions
\[
	\zeta'(x) = - \wp(x), \qquad \zeta(x)  = {1 \over x} + O(x).
\]

\emph{Weierstrass function $\sigma(x)$} is an entire function on $\mathbb{C}$ determined by the conditions
\[
	\left(\ln \sigma(x)\right)' = \zeta(x), \qquad \sigma(x) = x  + O(x^2).
\]

The functions $\zeta(x)$ and $\sigma(x)$ are odd and quasiperiodic:
\begin{equation} \label{sz}
  \zeta(x + 2 \omega_k) = \zeta(x) + 2 \eta_k,\qquad \sigma(x +  2 \omega_k) = - e^{2 \eta_k (x + \omega_k)} \sigma(x),
\end{equation}
where $\eta_k = \zeta(\omega_k)$ and the \emph{Legendre identity}
$2 \eta_1 \omega_2 - 2 \eta_2 \omega_1 = \pi i$ holds.

\emph{Baker--Akhiezer function} \cite{Krichever-90} is defined by the expression
\[
 \Phi(x; z) = {\sigma(z - x) \over \sigma(x) \sigma(z)} e^{\zeta(z) x}.
\]
Further we omit the parameter $z$ for brevity and write $\Phi(x)$ instead of $\Phi(x; z)$.

From \eqref{sz} we get
$\Phi(x + 2 \omega_k) = \Phi(x) \exp{(2 \zeta(z) \omega_k - 2 \eta_k z)}$. 

The Baker--Akhiezer function is a solution of the \emph{Lame equation} (see \cite{WW, Krichever-90})
\begin{equation} \label{Lame}
\Phi''(x) = 2 \wp(x) \Phi(x) + \wp(z) \Phi(x).
\end{equation}

The logarithmic derivative of the Baker--Akhiezer function is an elliptic function and we have 
(see \cite{WW}, 20.53, Example 2)
\begin{equation} \label{log}
{\Phi'(x) \over \Phi(x)} = {1 \over 2} {\wp'(x) + \wp'(z) \over \wp(x) - \wp(z)}.
\end{equation}

\begin{lem} \label{lem1}
The differential equation holds
\begin{equation} \label{Phieq}
 \Phi(x) \Phi'''(x) - 3 \Phi'(x) \Phi''(x) = - 6 \wp(z) \Phi(x) \Phi'(x) - 2 \wp'(z) \Phi(x)^2.
\end{equation} 
\end{lem}
\begin{proof}
The equation \eqref{Lame}, its derivative in $x$, and the equation \eqref{log} give the system
\begin{align*}
&\Phi''(x) = 2 \wp(x) \Phi(x) + \wp(z) \Phi(x), \\
&\Phi'''(x) = 2 \wp'(x) \Phi(x) + (2 \wp(x) + \wp(z)) \Phi'(x), \\
&2 (\wp(x) - \wp(z)) \Phi'(x) = (\wp'(x) + \wp'(z)) \Phi(x). 
\end{align*}
By expressing $\wp(x)$ and $\wp'(x)$ from the
first and second equation respectively, we obtain
\begin{align*}
\wp(x) &= {\Phi''(x) - \wp(z) \Phi(x) \over 2 \Phi(x)}, & \wp'(x) &= {\Phi'''(x) \Phi(x) - \Phi''(x) \Phi'(x) \over 2 \Phi(x)^2}.
\end{align*}
By substituting these two relations into the third equation, we get \eqref{Phieq}.
\end{proof}

\section{Buchstaber formal group}

Let $R$ be a commutative ring with unit.

A commutative one-dimensional \emph{formal group} over $R$ is a formal series $F(u,v) \in R[[u,v]]$
satisfying the conditions
\begin{align*}
 F(u, 0) &= u, & F(F(u,v), w) &= F(u, F(v,w)), & F(u, v) &= F(v,u). 
\end{align*}
See \cite{Haz} for the general theory of formal groups.
A modern approach and applications to Hirzebruch genera can be found in \cite{BP}.

An \emph{exponential} of the formal group $F(u, v)$ is a formal series
$f(x) \in R \otimes \mathbb{Q}[[x]]$ with initial conditions $f(0)=0, f'(0) = 1$ and addition law
\begin{equation} \label{f}
f(x+y) = F(f(x),f(y)).
\end{equation}
A \emph{logarithm} $g(u)$ is a series, functionally inverse to $f(x)$.
Over rings $R$ without torsion the formal group (as series with coefficients in $R$) can be recovered from its exponential
(a series with coefficients in $R \otimes \mathbb{Q}$).

A \emph{Buchstaber formal group} over $R$ is a formal group of the form \eqref{Buch}
\[
F(u,v) =\frac{u^2 A(v) - v^2 A(u)}{u B(v) - v B(u)},
\]
where $A(u), B(u) \in R[[u]]$ and $A(0) = B(0) = 1$.
Set $A(u) = 1 + \sum_k A_k u^k$, $B(u) = 1 + \sum_k B_k u^k$.
Note that the right hand side of \eqref{Buch} does not depend on the coefficients $A_2$ and $B_1$.
Therefore, we assume $A_2 = B_1 = 0$.

For the definition of \emph{universal formal group of a given form} see \cite{BBU, BU}. 
The ring of coefficients of the universal formal group of the form \eqref{Buch} is described in~\cite{BU}. 

In \cite{Buc90} it was shown that for the Buchstaber universal formal group the exponential is \eqref{fKr}.
We have (see \cite{BP, BBKrich})
$f_{Kr}(x) \in \mathbb{Q}[\alpha, \wp(z), \wp'(z), g_2][[x]]$.

\begin{cor} \label{Kr}
The function $f_{Kr}(x)$ is four-parametric with parameters $\alpha, \beta, \gamma, \lambda$,
where $\beta = \wp(z)$, $\gamma = \wp'(z)$, $\lambda = g_2$.
In this parameters
\begin{multline}\label{mith}
f_{Kr}(x) = x + \alpha x^2 + (\alpha^2 + \beta) {x^3 \over 2}
+ (\alpha^3 + 3 \alpha \beta - \gamma) {x^4 \over 3!} + \\+
(5 \alpha^4  + 30 \alpha^2 \beta + 45 \beta^2 - 20 \alpha \gamma - 3 \lambda) {x^5 \over 5!} + O(x^6), 
\end{multline}
$g_3 = 4 \beta^3 - \lambda \beta - \gamma^2$
and $\Delta = \lambda^3 - 27 (4 \beta^3 - \lambda \beta - \gamma^2)^2$.
\end{cor}

\begin{cor}[from Lemma \ref{lem1}] \label{corcor}
The function \eqref{fKr} with parameters $\alpha, \beta, \gamma, \lambda$ satisfies the differential equation
\begin{equation} \label{feq}
f(x) f'''(x) - 3 f'(x) f''(x) = C_1 f'(x)^2 + C_2 f(x) f'(x) + C_3 f(x)^2, 
\end{equation}
where
$C_1 = - 6 \alpha$, $C_2 = 6 \alpha^2 - 6 \beta$, $C_3 = 2 \gamma + 6 \alpha \beta - 2 \alpha^3$.
\end{cor}

\begin{proof}
We have $f(x) = {\exp(\alpha x) \over \Phi(x)}$.
The expressions for derivatives are
\begin{align*}
f'(x) &= {\exp(\alpha x) \over \Phi(x)^2} \left(\alpha \Phi(x) - \Phi'(x)\right),\\
f''(x) &= {\exp(\alpha x) \over \Phi(x)^3} \left(\alpha^2 \Phi(x)^2 - 2 \alpha \Phi(x) \Phi'(x) + 2 \Phi'(x)^2 - \Phi(x) \Phi''(x)\right),\\
f'''(x) &= {\exp(\alpha x) \over \Phi(x)^4} \left(\alpha^3 \Phi(x)^3 - 3 \alpha^2 \Phi(x)^2 \Phi'(x) + 6 \alpha \Phi(x) \Phi'(x)^2 - 6 \Phi'(x)^3 + \right. \\
 & \left. \qquad \qquad \qquad \qquad + 6 \Phi(x) \Phi'(x) \Phi''(x) - 3 \alpha \Phi(x)^2 \Phi''(x) - \Phi(x)^2 \Phi'''(x) \right).
\end{align*}
We substitute these expressions into \eqref{feq} and reduce both sides by the factor ${\exp(2 \alpha x) \over \Phi(x)^5}$. 
After substituting $\beta = \wp(z)$, $\gamma = \wp'(z)$, the right hand side of the equation gives
\begin{multline*}
- 6 \alpha \Phi(x) \left(\alpha \Phi(x) - \Phi'(x)\right)^2 + 6 (\alpha^2- \wp(z)) \Phi(x)^2 \left(\alpha \Phi(x) - \Phi'(x)\right) + \\
+ 2 ( \wp'(z) + 3 \alpha \wp(z) - \alpha^3) \Phi(x)^3
= \\ =
- \Phi(x) \left( 2 (\alpha^3 - \wp'(z)) \Phi(x)^2 - 6 (\alpha^2 + \wp(z)) \Phi(x) \Phi'(x) + 6 \alpha \Phi'(x)^2 \right).
\end{multline*}

The left hand side of the equation gives
\begin{multline*}
\alpha^3 \Phi(x)^3 - 3 \alpha^2 \Phi(x)^2 \Phi'(x) - 3 \alpha \Phi(x)^2 \Phi''(x) + 6 \alpha \Phi(x) \Phi'(x)^2 - \\
- \Phi(x)^2 \Phi'''(x) + 6 \Phi(x) \Phi'(x) \Phi''(x) - 6 \Phi'(x)^3 - \\
- 3 (\alpha \Phi(x) - \Phi'(x)) (\alpha^2 \Phi(x)^2 - 2 \alpha \Phi(x) \Phi'(x) + 2 \Phi'(x)^2 - \Phi(x) \Phi''(x)) = \\ =
- \Phi(x) (\Phi(x) \Phi'''(x) - 3 \Phi'(x) \Phi''(x) + 2 \alpha^3 \Phi(x)^2 - 6 \alpha^2 \Phi(x) \Phi'(x) + 6 \alpha \Phi'(x)^2).
\end{multline*}
By equating and dividing by $- \Phi(x)$, we obtain the relation \eqref{Phieq}. 
\end{proof}

\begin{lem} \label{lemlem}
Two solutions of \eqref{feq} with initial conditions $f(0) = 0$, $f'(0) = 1$ coincide,
if initial terms of their expansions as power series in $x$
up to $x^5$ coincide.
\end{lem}

\begin{proof}
Let
 \[
  f(x) = x + f_1 x^2 + f_2 x^3 + f_3 x^4 + f_4 x^5 + \ldots + f_k x^{k+1} + \ldots
 \]
From initial terms of the series expansion of equation \eqref{feq} we get
\[
 C_1 = - 6 f_1, \quad C_2 = 12 (f_1^2 - f_2), \quad C_3 = -12 (f_1^3 - 2 f_1 f_2 + f_3).
\]
Further at $x^{k-1}$ for $k \geqslant 4$ we obtain expressions of the form
\[
(k+1) k (k-4) f_k = P_k(f_1, \ldots, f_{k-1}),
\]
where $P_k$ are polynomials. This gives expressions for $f_k$ in $f_1, \ldots, f_{k-1}$ for $k \geqslant 5$.
\end{proof}

\begin{lem} 
For the formal group \eqref{Buch} with exponential $f(x)$ the relations hold
\begin{align} \label{AB}
B(f(x)) &= f'(x) - A_1 f(x),\\
2 A(f(x)) &=  2 f'(x)^2 - f(x) f''(x) - A_1 f(x) f'(x) - 2 B_2 f(x)^2. \nonumber
\end{align}
\end{lem}

\begin{proof}
From \eqref{Buch} we have
\begin{align*}
\left. {\partial F(u,v) \over \partial v} \right|_{v = 0} &= B(u) + A_1 u,\\
{u \over 2} \left. {\partial^2 F(u,v) \over \partial v^2} \right|_{v = 0} &= B(u)^2 - A(u) + A_1 u B(u) - B_2 u^2.
\end{align*}

On the other hand, from \eqref{f} we get
\begin{align*}
\left. {\partial F(u,v) \over \partial v} \right|_{v = 0} &=  f'(g(u)), \\
\left. {\partial^2 F(u,v) \over \partial v^2} \right|_{v = 0} &= f''(g(u)) + f'(g(u)) g''(0).
\end{align*}

Comparing these expressions, we obtain $g''(0) = - A_1$, $g'''(0) = 2 A_1^2 - 2 B_2$ and expressions \eqref{AB}.

Note that analogous arguments for the next derivative give the differential equation \eqref{feq}
with $C_1 = - 3 A_1$, $C_2 = A_1^2 - 4 B_2$, $C_3 = 6 (A_3 - B_3)$.
\end{proof}

\section{Isogenies of elliptic curves}

Two elliptic curves with Jacobians $\mathcal{E}_1$ and $\mathcal{E}_2$ are isogenic, 
if there exists an epimorphism $\mathcal{E}_1 \to \mathcal{E}_2$.
For a sublattice $\langle 2 \widetilde{\omega}_1, 2 \widetilde{\omega}_2 \rangle$ of $\langle 2 \omega_1, 2 \omega_2 \rangle$
we obtain the isogenie
$\mathbb{C}/\langle 2 \widetilde{\omega}_1, 2 \widetilde{\omega}_2 \rangle \to \mathbb{C}/\langle 2 \omega_1, 2 \omega_2 \rangle$.
The classical approach to elliptic curves isogenies see, for example, in \cite{Iso}.

In this work we will use the approach below to the description of the connection of invariants $g_2, g_3$
of isogenic curves and functions on them.
We use a construction from \cite{WW}, 20.51. According to this construction any elliptic function $f(x)$
can be presented as a rational function in $\wp(x)$ and $\wp'(x)$ with the same periods.

Let $f(x)$ be an elliptic function with periods $2 \omega_1, 2 \omega_2$.
The parallelogram with vertices $0$, $2 \omega_1$, $2 \omega_1 + 2 \omega_2$, $2 \omega_2$ we call the \emph{periods parallelogram}
with generators $(2 \omega_1, 2 \omega_2)$.

Consider the Weierstrass function $\wp(x; g_2, g_3)$ with lattice of periods $\langle 2 \omega_1, 2 \omega_2 \rangle$.
On the periods parallelogram with generators $(2 \omega_1, 2 \omega_2)$ it has order $2$ and a double pole at the origin.
Consider the same function as an elliptic function with periods $(2 \widetilde{\omega}_1, 2 \widetilde{\omega}_2)$, 
where $\langle 2 \widetilde{\omega}_1, 2 \widetilde{\omega}_2 \rangle$ is a sublattice of index $2$ of the lattice $\langle 2 \omega_1, 2 \omega_2 \rangle$.
On the periods parallelogram with generators $(2 \widetilde{\omega}_1, 2 \widetilde{\omega}_2)$ the function $\wp(x; g_2, g_3)$ has order $4$,
double poles at the origin and some other point. So, it can be expressed as
\[
 \wp(x; g_2, g_3) = {(\wp(x; \widetilde{g}_2, \widetilde{g}_3) - a_1) (\wp(x; \widetilde{g}_2, \widetilde{g}_3) - a_2) \over
 \wp(x; \widetilde{g}_2, \widetilde{g}_3) - b_1}
\]
for some $a_1, a_2, b_1$ such that $b_1 \ne a_1$, $b_1 \ne a_2$,
where $\wp(x; \widetilde{g}_2, \widetilde{g}_3)$ is the Weierstrass function
with lattice of periods $\langle 2 \widetilde{\omega}_1, 2 \widetilde{\omega}_2 \rangle$.

From the series expansion of this equation in $x$ at the origin, we obtain the relations on the parameters
\begin{align}
 b_1 &= a_1 + a_2, & &  \nonumber \\
 g_2 &= 4 (a_1 + 3 a_2) (3 a_1 + a_2), &  \widetilde{g}_2 &= g_2 - 20 a_1 a_2, \nonumber \\
 g_3 &= - 8 (a_1 + a_2) (a_1 - a_2)^2, &  \widetilde{g}_3 &= g_3 - 28 a_1 a_2 b_1, \label{iso} \\
 \Delta &= 256 a_1 a_2 (9 b_1^2 - 4 a_1 a_2)^2, & \widetilde{\Delta} &= 16 a_1^2 a_2^2 (9 b_1^2 - 4 a_1 a_2). \nonumber
\end{align}

Thus, we have expressed the invariants $g_2, g_3$ of the lattice $\langle 2 \omega_1, 2 \omega_2 \rangle$
and the invariants $\widetilde{g}_2, \widetilde{g}_3$ of the sublattice
$\langle 2 \widetilde{\omega}_1, 2 \widetilde{\omega}_2\rangle$ of index $2$ of the lattice $\langle 2 \omega_1, 2 \omega_2 \rangle$
in terms of parameters $a_1$, $a_2$.
In the following paragraphs, we will use similar considerations expressing 
one and the same function in terms of Weierstrass functions with different invariants.

Note that this construction as a way to express the invatiants
$\widetilde{g}_2, \widetilde{g}_3$ in terms of invariants $g_2, g_3$
is ambiguous. Namely, for given $g_2, g_3$ formulas \eqref{iso} determine the parameters $a_1, a_2$ ambiguously.
On the other hand, for a given lattice $\langle 2 \omega_1, 2 \omega_2 \rangle$ there are several sublattices
$\langle 2 \widetilde{\omega}_1, 2 \widetilde{\omega}_2\rangle$ of index $2$, namely, the lattices
$\langle 2 \omega_1, 4 \omega_2 \rangle$, $\langle 4 \omega_1, 2 \omega_2 \rangle$, $\langle 4 \omega_1, 2 \omega_1 + 2 \omega_2 \rangle$,
for which the invariants $\widetilde{g}_2, \widetilde{g}_3$ are different, but related with $g_2, g_3$ by formulas \eqref{iso}.

\section{Elliptic function of level 2}

The elliptic function of level $2$ is the Jacobi elliptic sine.
It can be determined as a solution of equation
\begin{equation} \label{Jac}
 f'(x)^2 = 1 - 2 \delta f(x)^2 + \varepsilon f(x)^4 
\end{equation}
with initial conditions $f(0)=0$, $f'(0) = 1$.

\begin{lem}
The Jacobi elliptic sine 
coincides with the two-parametric subfamily of functions $f_{Kr}(x)$, where the parameters from Corollary \ref{Kr}
are related to the parameters of the equation \eqref{Jac} by the relations
\begin{equation} \label{bl}
\alpha = 0, \quad \beta = - {2 \over 3} \delta, \quad \gamma = 0, \quad \lambda  = {16 \over 3} \delta^2 - 4 \varepsilon. 
\end{equation}
\end{lem}

\begin{proof}
According to Corollary \ref{corcor} and Lemma \ref{lemlem}, it is sufficient to prove, that the function $f_2(x)$,
determined as a solution of \eqref{Jac} with initial conditions $f(0)=0$, $f'(0) = 1$,
for parameters \eqref{bl} satisfies the differential equation \eqref{feq},
and initial terms of it's series expansion in $x$ up to $x^5$ coincide with \eqref{mith}.

From \eqref{Jac} we get
\[
 f(x) = x - {1 \over 3} \delta x^3 + {1 \over 30} (\delta^2 + 3 \varepsilon) x^5 + O(x^6).
\]
Taking into account relations \eqref{bl}, this coincides with \eqref{mith}.

For the given parameters equation \eqref{feq} takes the form
\begin{equation} \label{mir}
f(x) f'''(x) - 3 f'(x) f''(x) = 4 \delta f(x) f'(x). 
\end{equation}
By differentiating \eqref{Jac}, we obtain
$f''(x) = - 2 \delta f(x) + 2 \varepsilon f(x)^3$,
$f'''(x) = - 2 (\delta - 3 \varepsilon f(x)^2) f'(x)$.
Now substitution into \eqref{mir} gives identity.
\end{proof}

We further denote the elliptic function of level $2$ by $f_2(x)$.
It is non-degenerate for $\Delta = 64 \varepsilon^2 (\delta^2 - \varepsilon) \ne 0$.
For a non-degenerate elliptic function of level $2$ from \eqref{fKr} we have
\[
f_2(x) =  { \sigma(x) \sigma(\omega_1) \over \sigma(\omega_1 - x)} \exp( - \zeta(\omega_1) x).
\]

The elliptic function of level $2$ is a solution of the $N$-th special Hirzebruch functional equation \eqref{fe1} for even $N$.

It is the exponential of the universal formal group of the form
\[
 F(u,v)=\frac{u^2 -v^2}{u B(v)-v B(u)}, \quad B(0) = 1.
\]

\begin{lem} \label{lem52}
The non-degenerate elliptic function of level $n$ is odd if and only if $n = 2$.
\end{lem}

\begin{proof}
The elliptic function of level $2$, that is the Jacobi elliptic sine, is odd. This follows, for example, from equation \eqref{fe1} for $N=2$.
For $n > 2$ from the explicit form of \eqref{fKr} it follows that the set of poles of $f_n(x)$ is not symmetric relative to the origin,
therefore it can not be odd.
\end{proof}

The $4$-th special Hirzebruch functional equation \eqref{fe1} has the form
\begin{multline} \label{Hir3}
{ 1 \over f(x_2 - x_1) f(x_3 - x_1) f(x_4 - x_1)} + { 1 \over f(x_1 - x_2) f(x_3 - x_2) f(x_4 - x_2)} + \\
+ { 1 \over f(x_1 - x_3) f(x_2 - x_3) f(x_4 - x_3)} + { 1 \over f(x_1 - x_4) f(x_2 - x_4) f(x_3 - x_4)} = 0.
\end{multline}

In \cite{BN} it was shown that any solution of the $4$-th special Hirzebruch functional equation
is either an elliptic function of level $2$, or an elliptic function of level $4$.

\begin{cor} \label{cor22}
 Any odd solution of the $4$-th special Hirzebruch functional equation with initial conditions $f(0) = 0, f'(0) = 1$
 is an elliptic function of level $2$.
\end{cor}

As an application of Corollary \ref{cor22} we give several solutions of the problem
of expressing the elliptic function of level $2$ in terms of Weierstrass elliptic functions.

\begin{ex} On the periods parallelogram with generators $(2 \widetilde{\omega}_1, 2 \widetilde{\omega}_2) = (2 \omega_1, 4 \omega_2)$
the function $f_2(x)$ is an odd elliptic function of order $2$.
It has simple zeros at $0$, $2 \omega_2$ and simple poles at $\omega_1$, $\omega_1 + 2 \omega_2$. Therefore, the equality holds
\[
 f_2(x) = {- 2 (\wp(x; \widetilde{g}_2, \widetilde{g}_3) - c) \over \wp'(x; \widetilde{g}_2, \widetilde{g}_3)},
\]
where $\widetilde{g}_2, \widetilde{g}_3$ are invariants,
corresponding to the lattice $\langle 2 \widetilde{\omega}_1, 2 \widetilde{\omega}_2 \rangle$,
and $c = \wp(2 \omega_2; \widetilde{g}_2, \widetilde{g}_3)$.

The initial terms of the series expansion of this equality allows to find the relations between parameters included into it, namely, if
$g_2, g_3$ are invariants, corresponding to the lattice $\langle 2 \omega_1, 2 \omega_2 \rangle$, and $\beta = \wp(\omega_1; g_2, g_3)$, then
\begin{align*}
4 \beta^3 &- g_2 \beta - g_3 = 0, &  \widetilde{g}_2 &= - {1 \over 4} g_2 + 15 c^2, \\
 c &= - {1 \over 2} \beta, & \widetilde{g}_3 & = 4 c^3 - c \widetilde{g}_2.
\end{align*}
Likewise, for the same $\widetilde{g}_2, \widetilde{g}_3$ we have
\[
 f_2(x) = {\wp'(x) \over - 2 (\wp(x) - a) (\wp(x) - b)}, 
\]
where $a = \wp(\omega_1; \widetilde{g}_2, \widetilde{g}_3)$, $b = \wp(\omega_1 + 2 \omega_2; \widetilde{g}_2, \widetilde{g}_3)$.
We obtain the relations
\[
 c = - a - b, \quad \widetilde{g}_2 = 4 (a^2+ a b + b^2), \quad \widetilde{g}_2 = 4 a b c.
\]
\end{ex}

\begin{ex} On the periods parallelogram with generators $(2 \widetilde{\omega}_1, 2 \widetilde{\omega}_2) = (4 \omega_1, 4 \omega_2)$
the function $f_2(x)$ is an odd elliptic function of order $4$.
It has simple zeros at $0$, $2 \omega_1$, $2 \omega_2$, $2 \omega_1 + 2 \omega_2$
and simple poles at $\omega_1$, $3 \omega_1$, $\omega_1 + 2 \omega_2$, $3 \omega_1 + 2 \omega_2$.
Hence, it can be expressed as
\begin{align*}
 &f_2(x) = {\wp'(x) \over - 2 (\wp(x) - a) (\wp(x) - b)}, &
 &\text{where} & &\widetilde{g}_2 = 2 (a^2 + 4 a b + b^2), \\
 & & & & &\widetilde{g}_3 = - {1 \over 2} (a+b) (a^2 + 6 a b + b^2).
\end{align*}
\end{ex}

\begin{ex} \label{rem45} On the periods parallelogram with generators $(2 \widetilde{\omega}_1, 2 \widetilde{\omega}_2) = (2 \omega_1, 8 \omega_2)$
the function $f_2(x)$ is an odd elliptic function of order $4$.
It has simple zeros at $0$, $2 \omega_2$, $4 \omega_2$, $6 \omega_2$
and simple poles at $\omega_1$, $\omega_1 + 2 \omega_2$, $\omega_1 + 4 \omega_2$, $\omega_1 + 6 \omega_2$.
Hence, it can be expressed as
\begin{align*}
 &f_2(x) = { - 2 (\wp(x) - a)  (\wp(x) - b) \over \wp'(x)  (\wp(x) - c)}, & &\text{where} & & c = 2 b - a,\\
 & & & & &\widetilde{g}_2 = - 4 (a^2 - 2 a b - 2 b^2), \\
 & & & & &\widetilde{g}_3 =  4 b (a^2 - 2 a b - b^2).
\end{align*}
Here the relations on the parameters are obtained from initial terms of the series expansion of equation \eqref{Hir3} on this function.

This expression is important for us in terms of isogenies of curves.
The initial terms of the series expansion of this equation allow us to find the relations
\[
\wp(\omega_1, g_2, g_3) = 2 (b - 2 a), \quad g_2 = 4 (44 a^2 - 28 a b - 13 b^2), \quad g_3 = 8 (2 a - b) (28 a^2 - 12 a b - 17 b^2).
\]
Herewith
\[
 \Delta = 1024 (2 a+b) (a-b) (2 a - 5 b)^4, \qquad \widetilde{\Delta} = - 16 (2 a + b) (2 a - 5 b) (a-b)^4.
\]
\end{ex}

\section{Elliptic function of level 3}

The non-degenerate elliptic function of level $3$ has the form \eqref{fKr} 
for $z = {2 \over 3} k \omega_1, \quad k = 1,2$. After replacement of generators of the periods lattice
$(2 \omega_1, 2 \omega_2) \mapsto (- 2 \omega_1, 2 \omega_2)$ the points ${2 \over 3} \omega_1$ and ${4 \over 3} \omega_1$
are interchanged modulo the lattice,
so we can assume that the generators are chosen so that $k = 1$,  $z = {2 \over 3} \omega_1$.

We call the \emph{elliptic function of level $3$} the two-parametric subfamily of functions $f_{Kr}(x)$,
where the parameters (see Corollary \ref{Kr}) are related by
\begin{equation} \label{ag}
 \beta = 3 \alpha^2, \qquad \lambda = 12 \alpha (9 \alpha^3 + \gamma).
\end{equation}
The reason for this definition will be the results of Lemmas \ref{talem}, \ref{salem}.
Further we use the notation $f_3(x)$ for the elliptic function of level $3$. 

The non-degenerate elliptic function of level $3$ is a solution of the $N$-th special Hirzebruch functional equation \eqref{fe1} 
for $N$ divisible by $3$.
For $N=3$ we get the equation
\begin{equation} \label{Hir2}
{ 1 \over f(x_2 - x_1) f(x_3 - x_1)} + { 1 \over f(x_1 - x_2) f(x_3 - x_2)} + { 1 \over f(x_1 - x_3) f(x_2 - x_3)} = 0. 
\end{equation}

\begin{lem} \label{talem}
The parameters of a function of the form $f_{Kr}(x)$ (see Corollary \ref{Kr}), satisfying the equation \eqref{Hir2}, are related by \eqref{ag}.
\end{lem}
\textsc{The proof} consists in substituting the series \eqref{mith} into \eqref{Hir2}.

\begin{lem} \label{salem}
 For any $\alpha, \gamma$ such, that $\Delta = - 27 (8 \alpha^3 + \gamma) \gamma^3 \ne 0$, the function $f_3(x)$ with parameters $\alpha, \gamma$
 can be expressed in the form \eqref{fKr} with $z = {2 \over 3} \omega_1$.
\end{lem}

\begin{proof}
From Corollary \ref{Kr}, taking into account \eqref{ag}, we get
 \[
  g_2 = 12 \alpha (9 \alpha^3 + \gamma), \quad g_3 = - 216 \alpha^6 - 36 \alpha^3 \gamma - \gamma^2.
 \]
 Let $\langle 2 \omega_1, 2 \omega_2 \rangle$ be the periods lattice of Weierstrass $\wp$-function with invariants $g_2$, $g_3$.
 Since $\Delta \ne 0$, the lattice is non-degenerate. 
 In this conditions, the value $\wp(z) = 3 \alpha^2$ is a solution of the equation
 \[
  48 \wp(z)^4 - 24 g_2 \wp(z)^2 - 48 g_3 \wp(z) - g_2^2 = 0.
 \]
Since $\wp(z)$ is a function of order $2$, on the periods parallelogram with generators $(2 \omega_1, 2 \omega_2)$
this equation has not more than $8$ solutions $z$.
On the other hand, for a given lattice one can take any of the $8$ points of order $3$ as the pole $z$ of \eqref{fKr}
(the expression \eqref{fKr} with $z = {2 \over 3} \omega_1$
will be achieved by replacement of generators of the lattice).
Hence, for every such point and its corresponding expression \eqref{fKr}
the parameters $\alpha, \gamma$ correspond uniquely.
\end{proof}

\begin{thm}[\cite{BBZam, BBHi2}]
The elliptic function of level $3$ is the general solution of the $3$-d special Hirzebruch functional equation \eqref{Hir2} with initial conditions
$f(0) = 0, f'(0) = 1$.
\end{thm}

We will need initial terms of the series expansion of $f_3(x)$:
\begin{equation} \label{ryad3}
f_3(x) = x + \alpha x^2 + 2 \alpha^2 x^3 + {1 \over 6} (10 \alpha^3 - \gamma) x^4 + {1 \over 15} \alpha (22 \alpha^3 - 7 \gamma) x^5 + \ldots 
\end{equation}

\begin{thm}[(\cite{BB3}, see also \cite{BBU})]
The elliptic function of level $3$ is the 
exponential of the universal formal group of the form
\begin{equation} \label{F3!}
 F(u,v)=\frac{u^2 A(v) -v^2 A(u)}{u A(v)^2 - v A(u)^2}, \qquad A(0) = 1, \quad A''(0) = 0.
\end{equation}
\end{thm}

Note that the form \eqref{F3!} is a specialization of \eqref{Buch} with $B(u) = A(u)^2 - 2 A_1 u$.

The proof of this Theorem uses an expression of the elliptic function of level $3$ in terms of Weierstrass elliptic functions
(see \cite{BBU, BBZam, BB3}).
In our terms, such a function can be sought as a solution of equation \eqref{feq} that has an expression in Weierstrass $\wp$-functions,
order $3$ on the periods parallelogram with generators $(2 \widetilde{\omega}_1, 2 \widetilde{\omega}_2) = (2 \omega_1, 6 \omega_2)$,
and initial terms of the series expansion at the origin that coincide with the expansion \eqref{ryad3}.
Such a function is
\begin{align*}
 f(x) &= - 6 {\wp(x) + \alpha^2 \over 3 \wp'(x) + 6 \alpha \wp(x) - (2 \alpha^3 + \gamma)}, &
 \quad \text{where} \quad 3 \widetilde{g}_2 &= 4 \alpha (\alpha^3 - \gamma),\\
 & & - 27 \widetilde{g}_3 &= 8 \alpha^6 + 20 \alpha^3 \gamma - \gamma^2.
\end{align*}

\section{Elliptic function of level 4}

The non-degenerate elliptic function of level $4$ has the form \eqref{fKr} 
for $z = {1 \over 2} k \omega_1, \quad k = 1,3$. After replacement of generators of the periods lattice
$(2 \omega_1, 2 \omega_2) \mapsto (- 2 \omega_1, 2 \omega_2)$ the points ${1 \over 2} \omega_1$ and ${3 \over 2} \omega_1$
are interchanged modulo the lattice,
so we can assume that the generators are chosen so that $k = 1$,  $z = {1 \over 2} \omega_1$.

We call the \emph{elliptic function of level $4$} the two-parametric subfamily of functions $f_{Kr}(x)$, where the parameters (see Corollary \ref{Kr})
are related by
\begin{equation} \label{ab}
\gamma = 4 \alpha (4 \alpha^2 - 3 \beta), \quad \lambda = 4 (32 \alpha^4 - 24 \alpha^2 \beta + 3 \beta^2).
\end{equation}
The reason for this definition will be the results of Lemmas \ref{talem1}, \ref{talem2}, \ref{talem3}.
Further we use the notation $f_4(x)$ for the elliptic function of level $4$. 

The non-degenerate elliptic function of level $4$ is a solution of the $N$-th special Hirzebruch functional equation \eqref{fe1}
for $N$ divisible by $4$.
For $N=4$ we get equation~\eqref{Hir3}
\begin{multline*} 
{ 1 \over f(x_2 - x_1) f(x_3 - x_1) f(x_4 - x_1)} + { 1 \over f(x_1 - x_2) f(x_3 - x_2) f(x_4 - x_2)} + \\
+ { 1 \over f(x_1 - x_3) f(x_2 - x_3) f(x_4 - x_3)} + { 1 \over f(x_1 - x_4) f(x_2 - x_4) f(x_3 - x_4)} = 0.
\end{multline*}

\begin{lem} \label{talem1}
 For the non-degenerate elliptic function of level $4$ we have $\alpha \ne 0$.
\end{lem}

\begin{proof}
By substituting the series \eqref{mith} into \eqref{Hir3}, for $x_k \to 0$ we get
\[
   \gamma = 4 \alpha (4 \alpha^2 - 3 \beta).
\]
For the non-degenerate elliptic function of level $4$ we have $ \gamma = \wp'(z)$, $\beta = \wp(z)$ in \eqref{fKr}.
Hence, for $\alpha = 0$ we get $\wp'(z) = 0$.
But on the periods parallelogram with generators $(2 \omega_1, 2 \omega_2)$
we have $\wp'(x) = 0$ only in the points $x = \omega_1$, $x = \omega_2$, $x = \omega_1 + \omega_2$,
while $z = {1 \over 2} \omega_1$.
\end{proof}

\begin{lem} \label{talem2}
The parameters of a function of the form $f_{Kr}(x)$ (see Corollary \ref{Kr}), satisfying the equation \eqref{Hir3},
for $\alpha \ne 0$ are related by \eqref{ab}.
\end{lem}
\textsc{The proof} consists in substituting the series \eqref{mith} into \eqref{Hir3}.
Here, we need additionally the coefficient of the series \eqref{mith} at $x^6$, which is calculated by substituting \eqref{mith} into the equation
\eqref{feq}. It is equal to
\[
 {1 \over 5!} (\alpha^5 + 10 \alpha^3 \beta + 45 \alpha \beta^2 - 10 \alpha^2 \gamma - 22 \beta \gamma - 3 \alpha \lambda).
\]

\begin{lem} \label{talem3}
 For any $\alpha, \beta$ such that $\Delta = 256 \alpha^2 (5 \alpha^2 - 3 \beta) (4 \alpha^2 - 3 \beta)^4 \ne 0$,
 the function $f_4(x)$ with parameters $\alpha, \beta$
 can be expressed as \eqref{fKr} with $z = {1 \over 2} \omega_1$.
\end{lem}

\begin{proof}
 From Corollary \ref{Kr} taking into account \eqref{ab} we get
 \begin{equation} \label{g234}
 g_2 = 4 (32 \alpha^4 - 24 \alpha^2 \beta + 3 \beta^2), \quad g_3 = - 8 (2 \alpha^2 - \beta) (16 \alpha^4 - 8 \alpha^2 \beta - \beta^2). 
 \end{equation}
 Let $\langle 2 \omega_1, 2 \omega_2 \rangle$ be the periods lattice of Weierstrass $\wp$-function with invariants $g_2$, $g_3$.
 Since $\Delta \ne 0$, the lattice is non-degenerate. 
 In this conditions, the value $\wp(z) = \beta$ is a solution of the equation
\[
 64 \wp(z)^6 - 80 g_2 \wp(z)^4 - 320 g_3 \wp(z)^3 - 20 g_2^2 \wp(z)^2 - 16 g_2 g_3 \wp(z) + g_2^3 - 32 g_3^2 = 0.
\] 
Since $\wp(z)$ is a function of order $2$, on the periods parallelogram with generators $(2 \omega_1, 2 \omega_2)$
this equation has not more than $12$ solutions $z$.
On the other hand, for a given lattice one can take any of the $12$ points of order $4$ as the pole $z$ of \eqref{fKr}
(the expression \eqref{fKr} with $z = {1 \over 2} \omega_1$
will be achieved by replacement of generators of the lattice).
Hence, for every such point and its corresponding expression \eqref{fKr},
the parameters $\alpha, \beta$ correspond uniquely.
\end{proof}

\begin{cor}[from lemma \ref{lem52} and from \cite{BN}] \label{cor41}
 Any solution of the $4$-th special Hirzebruch functional equation \eqref{Hir3} with initial conditions $f(0) = 0, f'(0) = 1$
 that is not odd is the elliptic function of level $4$.
\end{cor}

We will need initial terms of the series expansion of $f_4(x)$:
\begin{equation} \label{ryad}
f_4(x) = x + \alpha x^2 + (\alpha^2 + \beta) {x^3 \over 2} - 5 \alpha (\alpha^2 - \beta) {x^4 \over 2} 
- (233 \alpha^4 - 186 \alpha^2 \beta - 3 \beta^2) {x^5 \over 40} + \ldots 
\end{equation}

\begin{lem}
On the periods parallelogram with generators $(2 \omega_1, 8 \omega_2)$ the function $f_4(x)$ can be expressed in the form
\begin{equation} \label{f4p1}
 f_4(x) =  { (\wp(x) - a_1) ( - {1 \over 2} \wp'(x) (\wp(x) - a_2) + \alpha (\wp(x) - a_3) (\wp(x) - a_4)) \over (\wp(x) - b_1) (\wp(x) - b_2) (\wp(x) - b_3) (\wp(x) - b_4)},
\end{equation}
where $a_k \ne b_k$.
\end{lem}

\begin{proof}
On the periods parallelogram with generators $(2 \omega_1, 8 \omega_2)$ the elliptic function $f_4(x)$ has order $4$,
zeros at $0$, $2 \omega_2$, $4 \omega_2$, $6 \omega_2$, and poles at
${1 \over 2} \omega_1$, ${1 \over 2} \omega_1 + 2 \omega_2$, ${1 \over 2} \omega_1 + 4 \omega_2$, ${1 \over 2} \omega_1 + 6 \omega_2$.

On the same periods parallelogram
the even elliptic function ${f_4(x) + f_4(-x) \over 2}$ has order $8$,
zeros at $0$, $2 \omega_2$, $4 \omega_2$, $6 \omega_2$, and two more points $p$ and $-p$.
At the points $0$, $4 \omega_2$ these zeros are double zeros.
It has simple poles at
${1 \over 2} \omega_1$, ${1 \over 2} \omega_1 + 2 \omega_2$, ${1 \over 2} \omega_1 + 4 \omega_2$, ${1 \over 2} \omega_1 + 6 \omega_2$,
${3 \over 2} \omega_1$, ${3 \over 2} \omega_1 + 2 \omega_2$, ${3 \over 2} \omega_1 + 4 \omega_2$, ${3 \over 2} \omega_1 + 6 \omega_2$.

The odd elliptic function ${f_4(x) - f_4(-x) \over 2}$ has order $8$ on the same periods parallelogram, zeros at
$0$, $2 \omega_2$, $4 \omega_2$, $6 \omega_2$, $\omega_1$, $\omega_1 + 4 \omega_2$, and two more points $q$, $- q$.
(Further we will get additionally that $\{q, - q\} = \{2 \omega_2, 6 \omega_2\}$,
and thus in the points $2 \omega_2$, $6 \omega_2$ the zeros are double zeros.)
It has simple poles in the same points as ${f_4(x) + f_4(-x) \over 2}$.

Hence, we have the form \eqref{f4p1} for $b_k = \wp({1 \over 2} \omega_1 + 2 (k-1) \omega_2)$,
$a_1 = \wp(2 \omega_2)$, $a_2 = \wp(q)$, $a_3 = \wp(4 \omega_2)$, $a_4 = \wp(p)$.
\end{proof}

\begin{lem}[on isogeny] \label{xxx}
Let the invariants $g_2, g_3$ of the lattice $\langle 2 \omega_1, 2 \omega_2 \rangle$ be expressed by formulas \eqref{g234}.
Then the invariants $\widetilde{g}_2, \widetilde{g}_3$ for one of the lattices
$\langle 2 \omega_1, 8 \omega_2 \rangle$, $\langle 8 \omega_1, 2 \omega_2 \rangle$, $\langle 8 \omega_1, 2 \omega_1 + 2 \omega_2 \rangle$
are expressed by formulas
\begin{equation} \label{tg}
 - 4 \widetilde{g}_2 = 13 \alpha^4 - 6 \alpha^2 \beta - 3 \beta^2, \quad 8 \widetilde{g}_3 = (\alpha^2 - \beta) (17 \alpha^4 - 14 \alpha^2 \beta + \beta^2).
\end{equation}
\end{lem}

\textsc{Proof} is obtained if in example \ref{rem45} we set
\[
 4 a = - 3 \alpha^2 + \beta, \qquad 2 b = \alpha^2 - \beta.
\]

\begin{thm} \label{thmain}
On the periods parallelogram with generators $(2 \omega_1, 8 \omega_2)$ we have
\begin{multline} \label{f4p}
f_4(x) = {1 \over 32} {\alpha  (4 \wp(x) + 3 \alpha^2 - \beta) (2 \wp(x) - \alpha^2+ \beta) (4 \wp(x) - 7 \alpha^2 + 5 \beta)
 \over \wp(x)^4+ c_1 \wp(x)^3 + c_2 \wp(x)^2 + c_3 \wp(x)  + c_4} - \\ - 
  {1 \over 32} { \wp'(x) (4 \wp(x) + 3 \alpha^2 - \beta)^2 \over \wp(x)^4+ c_1 \wp(x)^3 + c_2 \wp(x)^2 + c_3 \wp(x)  + c_4},
\end{multline}
where $c_1 = \alpha^2-\beta$, $c_2 = {51 \over 8} \alpha^4 - {21 \over 4} \alpha^2 \beta + {3 \over 8} \beta^2$,
$c_3 = - {1 \over 16} (\alpha^2 - \beta) (47 \alpha^4 - 34 \alpha^2 \beta - \beta^2)$,
$c_4 = {545 \over 256} \alpha^8 - {295 \over 64} \alpha^6 \beta + {435 \over 128} \alpha^4 \beta^2 - {55 \over 64} \alpha^2 \beta^3 + {1 \over 256} \beta^4$,
and the invariants $g_2$, $g_3$ of $\wp$-functions are expressed by formulas
\[
 g_2 = - {1 \over 4} (13 \alpha^4 - 6 \alpha^2 \beta - 3 \beta^2), \quad g_3 = {1 \over 8} (\alpha^2 - \beta) (17 \alpha^4 - 14 \alpha^2 \beta + \beta^2).
\]

\end{thm}

\begin{proof} 
The expression \eqref{f4p} satisfies the differential equation \eqref{feq}
for $C_1 = - 6 \alpha$, $C_2 = 6 (\alpha^2- \beta)$, $C_3 = 6 \alpha (5 \alpha^2 - 3 \beta)$.
For the right hand side of \eqref{f4p} 
initial terms of the series expansion at $x = 0$ up to $x^5$ coincide with \eqref{ryad}. 
Thus, according to Lemma \ref{lemlem} the Theorem is proven. See \S \ref{secproof} for the calculations.
\end{proof}

\begin{thm} \label{Tmain}
The elliptic function of level $4$ is the exponential of the universal formal group of the form
\[
 F(u,v)=\frac{u^2 A(v) -v^2 A(u)}{u B(v)-v B(u)}, 
\]
where $A(0) = B(0) = 1, B'(0) = 0, A''(0) = 0$, and the relation holds
\begin{equation} \label{relAB}
(2 B(u) + 3 A_1 u)^2 = 4 A(u)^3 - (3 A_1^2 - 8 B_2) u^2 A(u)^2. 
\end{equation}
\end{thm}

\begin{proof}
For a formal group \eqref{Buch} with exponential $f_4(x)$ in relations \eqref{AB} we have
$A_1 = 2 \alpha$, $- 2 B_2 = \alpha^2 - 3 \beta$.
Substituting these relations into \eqref{relAB} for $u = f(x)$ we obtain the relation
\begin{multline} \label{End}
4 (2 f'(x) + A_1 f(x))^2 = \\ = 
(4 f'(x)^2  - 2 f(x) f''(x) - 2 A_1 f(x) f'(x) - (3 A_1^2 - 4 B_2) f(x)^2) \cdot \\ \cdot
(2 f'(x)^2 - f(x) f''(x) - A_1 f(x) f'(x) - 2 B_2 f(x)^2)^2.
\end{multline}
This relation for $f(x) = f_4(x)$ follows from Theorem \ref{thmain}.
Namely, Theorem \ref{thmain} gives the expression \eqref{f4p} for $f_4(x)$ in terms of $\wp(x)$ and $\wp'(x)$.
The substitution of this expression into \eqref{End} taking into account the relations \eqref{wp} gives identity. More details in \S \ref{secproof}.

Thus, we have shown that for $f_4(x)$ the formal group has the form specified.
The universality results from the fact that the equation \eqref{End}
has a unique solution $f(x)$ with initial conditions $f(0) = 0$, $f'(0) = 1$ for given $A_1$, $B_2$.
Indeed, for $f(x) = x + \sum_{k=2}^\infty f_{k-1} x^k$ from \eqref{End} we obtain $A_1 = 2 f_1, 
B_2 = - 2 f_1^2 + 3 f_2$, and for $k \geqslant 3$ at $x^k$:
$(k+1) (3 k - 8) f_k = P_k(f_1, \ldots, f_{k-1})$
for some polynomials $P_k$.
Therefore $f(x) = f_4(x)$ with $2 \alpha = A_1$, $3 \beta = \alpha^2 + 2 B_2$.
\end{proof}

For topological applications it is important to study the coefficient ring of the formal group from theorem \ref{Tmain}.
In particular, the question to proove the absence of torsion in this ring remains open.
For elliptic functions of level $2$ and $3$ the coefficient rings of corresponding formal groups are described in \cite{BBU}.

\section{Computational proofs of theorems \ref{thmain} and \ref{Tmain} assertions} \label{secproof}

Listed below assertions from the proofs of theorems \ref{thmain} and \ref{Tmain} can be checked directly. 
Here $f_4(x)$ is given by formula \eqref{f4p} with the parameters $c_1, c_2, c_3, c_4, g_2, g_3$ determined by the statement of theorem \ref{thmain}.
Thus, $f_4(x)$ is determined as a rational function in $\wp(x)$ and $\wp'(x)$, depending on two parameters $\alpha$, $\beta$.
From Weierstrass equation \eqref{wp}
follow expressions for the higher derivatives of Weierstrass $\wp$-function in $\wp(x)$ and $\wp'(x)$:
\[
 \wp''(x) = 6 \wp(x)^2 - {1 \over 2} g_2, \quad \wp'''(x) = 12 \wp(x) \wp'(x), \quad \ldots 
\]
Hence the functions $f_4'(x), f_4''(x), f_4'''(x)$
are also expressed as rational functions in $\wp(x)$ and $\wp'(x)$, depending on two parameters $\alpha$, $\beta$.

Thus, assertions stated below are reduced to verification of relations
on rational functions in $\wp(x)$ and $\wp'(x)$ with relation \eqref{wp}.
However, direct calculations prove to be extremely cumbersome and difficult to verify.
Therefore each of these assertions has been verified by computer algebra system Maple 2015.
After each assertion we give the code checking this assertion.
A similar check can be done in many other computer algebra systems.

In the code below the function $\mathtt{f(x)} = f_4(x)$ should be defined by formula \eqref{f4p}, where $\wp(x) = \mathtt{WeierstrassP(x, g_2, g_3)}$,
$\wp'(x) = \mathtt{WeierstrassPPrime(x, g_2, g_3)}$, and the parameters  $c_1$, $c_2$, $c_3$, $c_4$, $g_2$, $g_3$
are defined according to the statement of Theorem \ref{thmain} in terms of parameters $\alpha$, $\beta$.

\begin{utv}
The expression \eqref{f4p} satisfies the differential equation \eqref{feq}
for $C_1 = - 6 \alpha$, $C_2 = 6 (\alpha^2- \beta)$, $C_3 = 6 \alpha (5 \alpha^2 - 3 \beta)$.
\end{utv}

\texttt{> simplify(}
$\mathtt{f(x) f'''(x) - 3 f'(x) f''(x)}$ \\
\text{ } \hspace{12.5pt} $\mathtt{- 6 ( - \alpha f'(x)^2 + (\alpha^2 - \beta) f(x) f'(x) + \alpha (5 \alpha^2 - 3 \beta) f(x)^2) = 0}$
\texttt{)}

\begin{utv}
For the right hand side of \eqref{f4p}
initial terms of the series expansion at $x = 0$ up to $x^5$ coincide with \eqref{ryad}. 
\end{utv}

\texttt{> simplify(series($\mathtt{f(x)}$, $\mathtt{x}$, $\mathtt{6}$))}

\begin{utv}
The expression \eqref{f4p} satisfies the differential equation \eqref{End}
for $A_1 = 2 \alpha$, $- 2 B_2 = \alpha^2 - 3 \beta$.
\end{utv}

\texttt{> simplify(}
$\mathtt{4 (2 f'(x) + 2 \alpha f(x))^2}$ \\
\text{ } \hspace{12.5pt} $\mathtt{- (4 f'(x)^2 - 2 f(x) f''(x) - 4 \alpha f(x) f'(x) - (14 \alpha^2 - 6 \beta) f(x)^2)}$\\
\text{ } \hspace{12.5pt} $\mathtt{(2 f'(x)^2 - f(x) f''(x) - 2 \alpha f(x) f'(x) + (\alpha^2 - 3 \beta) f(x)^2)^2 = 0}$
\texttt{)}

\end{document}